\documentclass{article}[11pt]
\usepackage{gastex}
\usepackage{tikz, hyperref}
\usepackage{amsmath,amscd}
\usepackage{amssymb}
\usepackage{upref}
\usepackage{multirow}
\usepackage{multicol}
\usepackage{url}
\usepackage[all]{xy}

\usepackage{color}
\usepackage{makeidx}
\makeindex
\usepackage{theorem}
\newtheorem{theorem}{Theorem}
\newtheorem{lemma}[theorem]{Lemma}
\newtheorem{proposition}[theorem]{Proposition}
\newtheorem{corollary}[theorem]{Corollary}
{\theorembodyfont{\rmfamily}%
  \newtheorem{example}[theorem]{Example}
   }
\newenvironment{proof}{\noindent\textit{Proof.}}
{\QED\vskip\theorempostskipamount} 
\newenvironment{proofof}[1]{\noindent\textit{Proof
    \protect{#1}.}}
                       {\QED\vskip\theorempostskipamount}
\def\petitcarre{\vrule height4pt width 4pt depth0pt}
\def\QED{\relax\ifmmode\eqno{\hbox{\petitcarre}}\else{%
  \unskip\nobreak\hfil\penalty50\hskip2em\hbox{}\nobreak\hfil
  \petitcarre
  \parfillskip=0pt \finalhyphendemerits=0\par\smallskip}
  \fi}
\newcommand\A{\mathcal{A}}

\newcommand\cL{\mathcal{L}}

\newcommand{\N}{\mathbb{N}}
\newcommand{\Z}{\mathbb{Z}}

\def\un(#1){\underline{#1}\,}

\DeclareMathOperator{\Card}{Card}

\definecolor{ivoire}{rgb}{0.99,0.99,0.8}

\definecolor{light-gray}{gray}{0.7}

\usepackage{calc}
\newcounter{hours}\newcounter{minutes}

\numberwithin{theorem}{section}
\numberwithin{equation}{section}
\numberwithin{figure}{section}
\numberwithin{table}{section}

\definecolor{lime}{HTML}{A6CE39}
\DeclareRobustCommand{\orcidicon}{%
	\begin{tikzpicture}
	\draw[lime, fill=lime] (0,0)
	circle [radius=0.16]
	node[white] {{\fontfamily{qag}\selectfont \tiny ID}};
	\draw[white, fill=white] (-0.0625,0.095)
	circle [radius=0.007];
	\end{tikzpicture}
	\hspace{-2mm}
}
\foreach \x in {A, ..., Z}{%
 \expandafter\xdef\csname orcid\x\endcsname{\noexpand%
 \href{https://orcid.org/\csname orcidauthor\x\endcsname}{\noexpand\orcidicon}}
}

\title{Sequential densities of rational languages}

\author{Alexi Block Gorman and Dominique Perrin}
\begin{document}

\maketitle

\begin{abstract}
We introduce the notion of  density of a rational language with respect to a sequence of probability measures. We prove that if $(\mu_n)$ is a sequence of Bernoulli measures converging to a positive Bernoulli measure $\overline{\mu}$, the sequential density is the ordinary density with respect to $\overline{\mu}$. We also prove that if $(\mu_n)$ is a sequence of invariant
probability measures converging in the strong sense to an invariant probability measure $\overline{\mu}$, then the sequential density of every rational language exists for this sequence.
\end{abstract}

\section{Introduction}
The density of rational languages with respect Bernoulli measures
has been widely studied (see \cite{BertheGouletOuelletNybergBroddaPerrinPetersen2024}
for references). 
It has been recently shown that it exists in the more general setting of invariant probability measures \cite{BertheGouletOuelletPerrin2025}.
In \cite{Lynch1993}, another generalization of the Bernoulli measure case was studied, namely the density of rational languages with respect to a sequence of Bernoulli measures.
This density is shown in \cite{Lynch1993} to always exist.

In this paper, we introduce the notion of sequential density of a rational language with respect to a sequence
of probability measures, thereby considering a setting that is more general than both
\cite{BertheGouletOuelletPerrin2025} and \cite{Lynch1993}.
We prove a result which describes the sequential density in the case of a sequence of Bernoulli measures
converging to a positive measure (Theorem~\ref{theoremDensityPositiveBernoulli}). 
This also provides a direct proof of Lynch's result in this setting.

We also prove the existence of the sequential density of every rational language,
with respect to a sequence $\mu$ of invariant probability measures converging strongly
to an invariant probability measure (Theorem~\ref{theoremMain}). This generalizes 
the main result of \cite{Lynch1993} in which only Bernoulli measures are considered.
The proof uses the main result of \cite{BertheGouletOuelletPerrin2025}.

The paper is organized as follows. In Section \ref{sectionPreliminaries},
we define the basic notions of symbolic dynamics and probability theory used in the paper.
In Section \ref{sectionSequential}, we define sequential densities.
We prove Theorem~\ref{theoremDensityPositiveBernoulli} in Section \ref{sectionDensityPositive}
and Theorem~\ref{theoremMain} in 
Section \ref{sectionStrong}. 
We give examples in Section~\ref{sectionWeak}
that demonstrate in which cases this density might not exist, or for which the density \emph{does} exist, but does not agree with the density of limiting measure to which the sequence converges. 
Finally, we consider in Section~\ref{sectionCombinatorial}
the notion of combinatorial density, which is a particular case of sequential density.

\section{Preliminaries}\label{sectionPreliminaries}

We provide a short introduction to key notions coming from symbolic dynamics
(see~\cite{LindMarcus1995} for more details).
\subsection{Shift spaces}
Let $A$ be finite set. We denote by $A^*$ the set of words on the alphabet $A$
and by $A^+$ the set of nonempty words.
A \emph{prefix code} is a subset  of $A^*$ which does
not contain a proper prefix of any of its elements.
Symmetrically, a \emph{suffix code} is a subset  of $A^*$ which does
not contain a proper suffix of any of its elements.

We consider the set $A^\Z$ of two-sided infinite sequences of elements of
$A$ as a topological space with the usual topology induced by the distance function
$d(x,y)=2^{-r(x,y)}$, where $r(x,y)=\min\{|n|\mid n\in\Z, x_n\ne y_n\}$.
For $x\in A^\Z$ and $i<j$, we denote $x_{[i,j)}	=x_i\cdots x_{j-1}$.
For $w\in A^*$, we denote $[w]=\{x\in A^\Z\mid x_{[0,|w|)}=w\}$
the cylinder with basis $w$.

A \emph{shift space} on $A$ is a closed and shift-invariant subset of $A^\Z$.
 Equivalently, $X$
is a shift space if there is a set $W$ of words on $A$, which we
call the set of \emph{forbidden blocks},
such that $X$ is the set of sequences $x$ without any occurrence of any word in $W$.

The language of $X$, denoted by $\cL(X)$, is the set of words $w$ 
which appear in some element of $X$. 
We denote by $\cL_k(X)$ the set of words of length $k$ in $\cL(X)$.
A \emph{shift of finite type} is a shift space defined by a finite set of forbidden blocks.
A shift space $X$  is called \emph{irreducible} if for every $u,w\in\cL(X)$
there exists $w$ such that $uwv\in\cL(X)$.

Let $X$ be a shift space and let $k\ge 1$ be an integer. Let $u\mapsto \langle u\rangle$
be a bijection from the set $\cL_k(X)$ onto an alphabet $A_k$. Let $\varphi_k\colon X\to A_k^\Z$ be
the map defined by $y=\varphi_k(x)$ if $y_n=x_{[n,n+k)}$ for every $n\in\Z$.
The image $\varphi_k(X)$ of $X$ under $\varphi_k$ is a shift space on $A_k$ called
the \emph{$k$-block presentation} of $X$.

\subsection{Rational languages}

A language on the alphabet $A$ is a subset of $A^*$.
A finite automaton $\A=(Q,i,T)$ on the alphabet $A$ is given by a map $(q,a)\mapsto q\cdot a$
from the finite set $Q\times A$ into $Q$. The map extends in a natural way to a right action of $A$
on $Q$. The language \emph{recognized} by the automaton $\A$ is the set 
\[L=\{w\in A^*\mid i\cdot w\in T\}.\]
A \emph{rational language} is a language which can be recognized by some finite automaton. 

A language $L$ is \emph{recognized} by a morphism $\varphi\colon A^*\to M$
from $A^*$ to a monoid $M$ if $L=\varphi^{-1}(N)$ for some $N\subset M$.
A language is recognizable if and only if it can be recognized by a morphism 
onto a finite monoid.

The language of a shift of finite type is a rational language.
As an equivalent definition, a language is rational if it can be obtained from
the finite subsets of $A^*$ by a finite number of unions, products (concatenations), and (kleene) star, where
the star of a language $L$ is \[L^*=\{w_1w_2\cdots w_n\mid w_i\in L, n\ge 0\}.\]
For example, the language $L=ab\{a,b\}^*$ of words on the alphabet $\{a,b\}$
beginning with $ab$ is rational.

A finite monoid $M$ is \emph{aperiodic} if there is an integer $n\ge 1$ such that
$m^{n}=m^{n+1}$ for every $m\in M$.
A rational language is \emph{aperiodic} if it can be recognized by a morphism into a finite aperiodic monoid.
For example, the language $(a^2)^*$ is rational, but one can show that it is not aperiodic.

Rational languages are definable in the monadic second order logic of the integers with addition,
and aperiodic languages in the corresponding first-order logic (see~\cite{book/Perrin2004}).

\subsection{Probability measures}
We introduce probability measures on shift spaces
(see \cite{BoylePetersen2011} for a more complete introduction).

A \emph{probability measure} on $A^\Z$ is a map $\mu$ assigning to every Borel subset $U \subseteq A^{\Z}$ a real number $\mu(U)$ which is \emph{$\sigma$-additive}
(that is, $\mu(\cup_{n\ge 0} U_n)=\sum_{n\ge 0} \mu(U_n)$ for every sequence $(U_n)$ of pairwise disjoint
Borel sets $U_n$) and $\mu(X)=1$.

It follows from $\sigma$-additivity that a measure is \emph{monotone}; that is,
$U\subset V$ implies $\mu(U)\le\mu(V)$ and that, if
$U_1\subset U_2\subset\ldots$, then $\mu(\cup_{i\ge 1} U_i)=\lim_{i\to\infty}\mu(U_i)$.
For $w\in A^*$, we denote $\mu(w)=\mu([w])$. 

For each $n\ge 0$, the map $\mu_n\colon A^n\to[0,1]$
defined by $\mu_n(w)=\mu(w)$ is a probability measure on $A^n$ called the $n$-th \emph{marginal distribution}
of $\mu$.
For a set $W\subset A^*$,
we denote by $\mu(W)$ the sum $\sum_{w\in W}\mu(w)$, which is a real number or $+\infty$.

A probability measure $\mu$ on $A^\Z$ is \emph{invariant} if $\mu(S^{-1}U)=\mu(U)$
for every Borel set $U$.
An invariant probability measure is \emph{ergodic} if for every Borel set $U$ invariant under the shift, one has $\mu(U)=0$ or $\mu(U)=1$. By the Birkhoff ergodic theorem, an invariant measure
$\mu$ is ergodic if and only if 
\begin{equation}
\lim_{n\to\infty}\frac{1}{n}\sum_{i=0}^{n-1}\mu(U\cap S^{-i}V)=\mu(U)\mu(V)\label{eqErgodic}
\end{equation}
for every pair $U,V$ of Borel sets.

Two distinct ergodic measures $\mu,\nu$
on $A^\Z$ are \emph{mutually singular}, that is, there
exists a Borel set $U$ such that
$\mu(U)=\nu(X\setminus U)=0$. Indeed,
let $f$ be a measurable function such
that $\int fd\mu\ne\int f d\nu$. Let
$U$ be the set of $x\in A^\Z$ such that
$\frac{1}{n}\sum_{i=0}^{n-1}f(x)$
converges to $\int fd\nu$. Then
$\mu(U)=\nu(X\setminus U)=0$.

An invariant probability measure $\mu$ is \emph{mixing} if \eqref{eqErgodic} holds without taking the average, that is, if
\begin{equation}
\lim_{n\to\infty}\mu(U\cap S^{-n}V)=\mu(U)\mu(V)\label{eqMixing}
\end{equation}
for every pair $U,V$ of Borel sets.

We consider the space of probability measures on $A^\Z$ as a topological
space with the \emph{weak-$*$ topology}, in which $\mu_n\to\mu$
if $\mu_n(U)\to\mu(U)$ for every Borel set $U$. By the Banach-Alaoglu theorem,
this space is compact. The limit of a sequence of invariant measures is invariant
and the limit of a sequence of ergodic measures is ergodic.

A sequence $(\mu_n)_{n\ge 1}$ of probability measures on $A^\Z$ \emph{converges in the
strong sense} to a measure $\overline{\mu}$ if for every $\varepsilon>0$ there is $N\ge 1$
such that for every Borel set $U$ and every $n\ge N$, one has
$|\mu_n(U)-\overline{\mu}(U)|\le\varepsilon$.

A Bernoulli measure is a probability measure $\mu$ on $A^\Z$
such that $\mu(uv)=\mu(u)\mu(v)$ for every $u,v\in A^*$.
A Bernoulli measure is ergodic. If $(\mu_n)_{n\ge 1}$
is a convergent sequence of Bernoulli measures, the limit
is a Bernoulli measure. Indeed, for every $u,v\in A^*$,
one has $\mu(uv)=\lim_{n\to\infty}\mu_n(uv)=\lim_{n\to\infty}\mu_n(u)\mu_n(v)=\mu(u)\mu(v)$.

A \emph{Markov measure} is a measure defined by a stochastic $A\times A$-matrix $P$ 
called its {transition matrix} and a vector $\pi$, called its \emph{initial vector},
by 
\[\mu(a_0\cdots a_{n-1})=\pi_{a_0}P_{a_0a_1}\cdots P_{a_{n-2}a_{n-1}}\]
for $a_0,\ldots,a_{n-1}\in A$.
If $\pi P=\pi$, the measure is invariant. An invariant  Markov measure defined by
an irreducible matrix $P$ is ergodic.

A \emph{$k$-step Markov measure} on $A^\Z$ is a measure $\mu$ such that $\mu(U)=\nu(\varphi_k(U))$
for every Borel set $U$, and
for some Markov measure $\nu$ on the $k$-block presentation $\varphi_k(A^\Z)$.

The \emph{support} of a probability measure $\mu$ on $A^\Z$ is the set of words $w$
such that $\mu(w)>0$.
For every irreducible shift of finite type,
there is an ergodic measure supported by $X$, called
its measure of maximal entropy. 

Let $M$ be an irreducible matrix with left eigenvector $v$ and right eigenvector $w$
for the dominant eigenvalue $\lambda$, with the normalization $v\cdot w=1$.
The \emph{measure of maximal entropy} $\mu$ 
on the shift of finite type defined by  $M$ is the Markov measure with initial vector $\pi_i=v_iw_i$
and transition matrix $P_{ij}=(w_j/\lambda w_i)M_{i,j}$ (see \cite[chapter 13]{LindMarcus1995}). 

\begin{example}\label{exampleMeasureMaxEntropy}
Let $X$ be the shift of finite type defined by the matrix 
\[M=\begin{bmatrix}1&0&1\\1&1&1\\1&1&1\end{bmatrix}.\]
Its dominant eigenvalue is $\lambda=(3+\sqrt{5})/2$.
The corresponding left and right eigenvectors $v$ and $w$, normalized by $v\cdot w=1$, are
\[v=\frac{1}{2\lambda-3}\begin{bmatrix}1&\lambda-2&1\end{bmatrix},\quad w=\begin{bmatrix}\lambda-2\\1\\1\end{bmatrix}.\]

Thus, $\mu_2$ is the Markov measure defined by
\[\pi=\frac{1}{2\lambda-3}\begin{bmatrix}\lambda-2&\lambda-2&1\end{bmatrix},\quad
P=\begin{bmatrix}\frac{1}{\lambda}&0&\frac{1}{\lambda-1}\\\frac{\lambda-2}{\lambda}&\frac{1}{\lambda}&\frac{1}{\lambda}\\
\frac{\lambda-2}{\lambda}&\frac{1}{\lambda}&\frac{1}{\lambda}\end{bmatrix}.\]
\end{example}
%%%%%%%%%%%%%%%%%%%%%%

\section{Sequential densities}\label{sectionSequential}

Let $\mu=(\mu_n)_{n\ge 0}$ be a sequence of probability measures
on $A^\Z$. Given a language $L$ on the alphabet $A$, the \emph{sequential density}
of $L$ with respect to $\mu$ is
\[\delta_\mu(L)=\lim_{n\to\infty}\frac{1}{n}\sum_{i=0}^{n-1}\mu_i(L\cap A^i)\]
whenever the limit exists. We say that $L$ has a density in the strong
sense if the limit $\lim_{n\to\infty}\mu_n(L\cap A^n)$ exists, and not only
in average.
When the sequence $(\mu_n)$ is constant, we meet the usual notion of density (see \cite{BertheGouletOuelletNybergBroddaPerrinPetersen2024}).

Let us list a few elementary properties of sequential densities, which are actually the same for ordinary densities.
If $L$ has a sequential density with respect to $\mu$, then $0\le \delta_\mu(L)\le 1$.

The sequential density if \emph{finitely additive}; that is, if $L,L'$ are disjoint languages
having a sequential density (resp. in the strong sense), then
$L\cup L'$ has a sequential density (resp. in the strong sense).

The sequential density is \emph{monotone}, that is, if $L$ and $L'$ have sequential densities (resp. in the
strong sense) and $L\subset L'$, then $\delta_\mu(L)\le\delta_\mu(L')$.

The following result  is from \cite{Lynch1993}.
\begin{theorem}\label{theoremLynch}
If  $\mu$ is a  convergent sequence of Bernoulli measures on $A^\Z$,
the sequential density of every rational language exists. If $L$ is aperiodic,
the density exists in the strong sense.
\end{theorem}
This result from \cite{Lynch1993} is stated
for $A=\{a,b\}$. 
This is not really a restriction
since every alphabet can be encoded with two letters.
Moreover, it is assumed that $\mu_n(a)$
does not converge to $1$. 
Exchanging
$a$ and $b$ allows us to always satisfy this hypothesis.
Finally, various hypotheses
are stipulated concerning the convergence
of the sequence $\mu_n(a)$. The various
cases (which appear in the example below)
cover all possibilities, and therefore
there is no need to mention these
cases in the statement.

We first give an example illustrating Lynch's theorem. For two positive real functions $f$ and $g$,
we write $f\ll g$ if  for every $\epsilon>0$
there is $n_0 \in \N$ such that $f(n)\le\epsilon g(n)$ for every $n \in \N_{\geq n_0}$.
We write $f\sim g$ if $f(n)/g(n)$ converges to $1$.
\begin{example}\label{exampleLynch}
Let $A=\{a,b\}$ and $L=A^*aA^*$. We have
\begin{equation}\label{eqExampleLynch}
\delta_\mu(L)=\begin{cases}0&\mbox{if $\mu_n(a)\ll n^{-1}$,}\\
                           1-e^{-c}&\mbox{if $\mu_n(a)\sim cn^{-1}$,}\\
                           1&\mbox{if $n^{-1}\ll \mu_n(a)$.}
                           \end{cases}
\end{equation}
We shall obtain these values below in Example~\ref{exampleIdeal}. Note that,
in the second case, the sequential density
exists, but fails to be equal to $\delta_{\overline{\mu}}$. 
 Indeed, in this case, $\overline{\mu}(a)=0$ and therefore $\delta_\mu(L)\ne\delta_{\overline{\mu}}(L)=0.$
\end{example}
The following example illustrates the case
where the limit exists only in average.
\begin{example}
Let $L=(A^2)^*$ be the set of words of even length. Then,
\[\mu_i(L\cap A^i)=\begin{cases}1&\mbox{if $i$ is even,}\\0&\mbox{otherwise.}\end{cases}\]
Therefore, the sequence $\mu_i(L\cap A^i )$
does not converge. One has of course
$\delta_\mu(L)=1/2$ for every sequence
of probability measures $\mu$.
\end{example}

%%%%%%%%%%%%%%%%%%%%
\section{A formula for sequential densities}\label{sectionDensityPositive}
The following result covers Theorem \ref{theoremLynch}
when $\mu$ converges to a positive measure.
It also proves a stronger property, since we prove that the sequential
density not only exists, but is equal to the density with respect to the limit measure $\overline{\mu}$ (we have seen in Example~\ref{exampleLynch} that this can fail
to hold if $\overline{\mu}$ is not positive). 
We give a direct proof.
\begin{theorem}\label{theoremDensityPositiveBernoulli}
Let $\mu=(\mu_n)_{n\ge 0}$ be a sequence of Bernoulli measures on $A^\Z$. If $\mu$ converges
to a positive Bernoulli measure $\overline{\mu}$, then for every rational language $L\subset A^*$, the sequential
density of $L$ is equal to its density with respect to $\overline{\mu}$.
The same holds for densities in the strong sense for an aperiodic language $L$.
\end{theorem}
\begin{proof}
Let $\A=(Q,i,T)$ be a deterministic automaton recognizing $L$. 
Since $L$ is a finite union of languages
recognized by $(Q,i,\{t\})$ for $t\in T$, we may assume that $T=\{t\}$. 
Without loss of generality,
we may also assume that every $q\in Q$ lies on a path from $i$ to $t$. For a Bernoulli measure $\nu$ on $A^\Z$,
let $P(\nu)$ be the stochastic $Q\times Q$-matrix
defined by 
\begin{equation}\label{eqMarkovChain}
P(\nu)_{p,q}=\sum_{a\in A,p\cdot a=q}\nu(a).\end{equation}

Since $\overline{\mu}$ is positive, the matrix $P(\overline{\mu})$ is irreducible.
Since $\mu$ converges to $\overline{\mu}$, we may assume that all matrices $P_n=P(\mu_n)$ have
period $p$ equal to the period of $P(\overline{\mu})$. We may assume $p=1$ and therefore
that each $P_n$ is primitive. 

Let $\lambda_{n,k}$, $1\le k\le d$ be the distinct eigenvalues of $P_n$ and let $\gamma_{n,k}$ be the multiplicity
of $\lambda_{n,k}$ in the minimal polynomial of $P_n$. We may assume that $\lambda_{n,1}=1$. Since $P_n$
is primitive, we have $\gamma_{n,1}=1$.

Using the Jordan normal form of $P_n$, we may write
\[P_n=U\begin{bmatrix}1&0&\cdots&0\\0&J_{n,2}&\cdots&0\\0&0&\ddots\\0&0&\cdots&J_{n,d}\end{bmatrix}U^{-1}\]
where for $2\le k\le d$, $J_{n,k}$ is the Jordan block of $P_n$ corresponding to $\lambda_{n,k}$.
The columns of the matrix $U$ are generalized eigenvectors of $P_n$.
In particular, the first column is the vector $v$ with all components equal to $1$
and its first row is the stochastic vector $w_n$ such that $w_nP_n=w_n$.
Since $P_n$ converges to $P(\overline{\mu})$, the vectors $w_n$ tend to the
stochastic vector $w$ such that $wP(\overline{\mu})=w$.

 Since $P(\overline{\mu})$
is primitive, there is $\epsilon>0$ such that all eigenvalues $\lambda_{n,k}$ for $k\ge 2$ satisfy $|\lambda_{n,k}|<1-\epsilon$. 
So each $J_{n,k}^n$ tends to $0$, and therefore,
\[\lim_{n\to\infty}P_n^{n}=\lim_{n\to\infty}vw_n=\lim P(\overline{\mu})^n.\]
 This implies that $\delta_\mu(L)=\delta_{\overline{\mu}}(L)$
for every rational language $L$.
\end{proof}

\begin{example}\label{exampleIdeal}
Let $A=\{a,b\}$ and $L=A^*aA^*$. The automaton $\A$ is represented in Figure~\ref{figureAutomatonA}.
\begin{figure}[hbt]
\centering
\tikzset{node/.style={circle,draw,minimum size=.6cm,inner sep=0pt}}
    \tikzset{title/.style={minimum width=.4cm,minimum height=.4cm}}
    \tikzstyle{every loop}=[->,shorten >=1pt,looseness=6]
    \tikzstyle{loop left}=[in=130,out=220,loop]
    \tikzstyle{loop right}=[in=-45,out=45,loop]
    \tikzstyle{loop below}=[in=-45,out=-135,loop]
    \tikzstyle{loop above}=[in=135,out=45,loop]
\begin{tikzpicture}
\node[node](1)at(0,0){$1$};\node[node](2)at(2,0){$2$};

\draw[->,loop above,above](1)edge node{$b$}(1);
\draw[->,above](1)edge node{$a$}(2);
\draw[->,above,loop above](2)edge node{$a,b$}(2);
\draw[->](-.5,0)-- node{}(1);\draw[->](2)-- node{}(2.5,0);
\end{tikzpicture}
\caption{The automaton $\A$.}\label{figureAutomatonA}
\end{figure}
The matrix $P_n$ is
\[P_n=\begin{bmatrix}q_n&p_n\\0&1\end{bmatrix}\]
and we have
\[P_n^n=\begin{bmatrix}1&1\\1&0\end{bmatrix}\begin{bmatrix}1&0\\0&q_n\end{bmatrix}^n\begin{bmatrix}0&1\\1&-1\end{bmatrix}
=\begin{bmatrix}0&1\\0&1\end{bmatrix}+q_n^n\begin{bmatrix}1&-1\\0&0\end{bmatrix}.\]
So $\delta_\mu(L)=1-\lim_{n\to\infty}q_n^n$.

This gives the values of $\delta_\mu(L)$ indicated in Equation~\eqref{eqExampleLynch}.

\end{example}

The following result is due, for a single Bernoulli measure, to Sch\"utzenberger \cite{Schutzenberger1965}
(see also ~\cite[Theorem 13.4.5]{BerstelPerrinReutenauer2009}).

The statement uses elementary notions for finite monoids (see~\cite{BerstelPerrinReutenauer2009}). In particular, a finite monoid has a unique minimal ideal, which is the union of the minimal right ideals of $M$, and also the union of the minimal left ideals. For every minimal right ideal $R$ and every minimal left ideal $L$, the intersection $R\cap L$ is a finite group. 
These groups are all isomorphic.
If $M$ is aperiodic, then $G$ is trivial.

\begin{corollary}\label{corollaryEquidistribution}
Let $\mu=(\mu_n)_{n\ge 0}$ be a sequence of Bernoulli measures on $A^\Z$ converging to a positive
Bernoulli measure $\overline{\mu}$. Let $\varphi\colon A^*\to M$
be a morphism onto a finite monoid. Let $K$ be the minimal ideal of $M$. Set $\nu=\delta_\mu\circ\varphi^{-1}$.
Then, for every $m\in M$, one has the following.
\begin{enumerate}
\item $\nu(m)>0$ if and only if $m\in K$.
\item For every $m\in K$, one has
\begin{equation}
\nu(m)=\frac{\nu(mM)\nu(Mm)}{d},
\end{equation}
where $d=\Card(mM \cap Mm)$.
\end{enumerate}
If $M$ is aperiodic, then $d=1$ and $\delta_\mu(L)$ exists in the strong sense.
\end{corollary}
\begin{proof}
By Theorem~\ref{theoremDensityPositiveBernoulli}, we have $\delta_\mu=\delta_{\overline{\mu}}$.
Since the above result holds for a positive Bernoulli measure, the statement follows.
\end{proof}
\begin{example}
Let us consider again $L=A^*aA^*$ as in Examples \ref{exampleLynch} and \ref{exampleIdeal}.
The case $n^{-1}\ll p_n$ occurs if $\overline{\mu}$ is positive. The value $\delta_\mu(L)=1$
is consistent with Corollary \ref{corollaryEquidistribution}. 
Indeed, in this case the monoid $M$ has two elements $0,1$ with $\varphi(a)=0$ and $\varphi(b)=1$.
We have $K=\{0\}$ and $\delta_\mu(L)=\nu(K)=1$.

In the cases $\mu_n(a)\ll n^{-1}$ or $\mu_n(a)\sim cn^{-1}$, we have $\nu(K)<1$.
\end{example}
Theorem~\ref{theoremDensityPositiveBernoulli}
can be generalized to the case
of a sequence of Markov measures
converging to a positive Markov measure.
Indeed, let $\nu$ be a  
 Markov measure defined by an initial
vector $v$ and a stochastic
$A\times A$-matrix $M$. 
Then, we replace the matrix $P(\nu)$ of Equation~\ref{eqMarkovChain}
by the $(Q\times A)\times (Q\times A)$ matrix
\[P(\nu)_{(p,a),(p\cdot a,b)}=M_{a,b}.\]

%%%%%%%%%%%%%%%%%%%%

\section{Convergence in the strong sense}\label{sectionStrong}

\begin{theorem}\label{theoremMain}
Let $\mu=(\mu_n)_{n\ge 0}$ be a sequence of invariant probability measures converging in the strong sense
to an invariant probability measure $\overline{\mu}$. Then, for every rational language $L$,
the sequential density of $L$ exists and 
$\delta_\mu(L)=\delta_{\overline{\mu}}(L)$. If $L$ is aperiodic, the sequential density exists in the strong sense.
\end{theorem}
We use the following result, established in \cite{BertheGouletOuelletPerrin2025}. 
\begin{theorem}\label{theoremBGOP}
Let $\mu$ be an invariant probability measure. Every rational language $L$ has a density
with respect to $\mu$. If $L$ is aperiodic, the density exists in the strong sense.
\end{theorem}
We also use the following elementary result from real analysis (see \cite[Section 8.20]{Apostol1974}).
We say that a sequence $(u_{m,n})$ of real numbers converges to $a$ if for every $\varepsilon>0$,
there is $N \in\N$ such that $|u_{m,n}-a|\le\varepsilon$ whenever $m,n\ge N$. 
If the sequence converges to $a$,
then the limits $\lim_{m\to\infty}\lim_{n\to\infty}u_{m,n}$ and $\lim_{n\to\infty}\lim_{m\to\infty}u_{m,n}$
also converge to $a$. 
The converse is also true with the additional hypothesis of uniform convergence.

\begin{lemma}\label{lemmaDoubleSequences}
Let $(u_{m,n})_{m,n\ge 0}$ be a sequence of real numbers. If
\begin{enumerate}
\item[\rm(i)] $\lim_{m\to\infty}\lim_{n\to\infty}u_{m,n}=a$,
\item[\rm(ii)] $\lim_{n\to\infty}u_{m,n}$ exists uniformly in $m$,
\end{enumerate}
then $(u_{m,n})$ converges to $a$.
\end{lemma}

\begin{proofof}{of Theorem~\ref{theoremMain}}
Set 
\[u_{m,n}=\frac{1}{m}\sum_{i=0}^{m-1}\mu_n(L\cap A^i).\]
 We have by Theorem \ref{theoremBGOP},
\[\lim_{m\to\infty}\lim_{n\to\infty}u_{m,n}=\lim_{m\to\infty} \frac{1}{m}\sum_{i=0}^{m-1}\overline{\mu}(L\cap A^i)=
\delta_{\overline{\mu}}(L).
\]
On the other hand, 
\[\lim_{n\to\infty}u_{m,n}=\lim_{n\to\infty}\frac{1}{m}\sum_{i=0}^{m-1}\mu_n(L\cap A^i)=\frac{1}{m}\sum_{i=0}^{m-1}\overline{\mu}(L\cap A^i),
\]
 uniformly in $m$ since the sequence $(\mu_n)$ converges strongly.
Therefore, by Lemma \ref{lemmaDoubleSequences}, the sequence $(u_{m,n})$ converges to $\delta_{\overline{\mu}}(L)$.

If $L$ is aperiodic, set $v_{m,n}=\mu_n(L\cap A^m)$. By Theorem~\ref{theoremBGOP}, we have
\[\lim_{m\to\infty}\lim_{n\to\infty} v_{m,n}=\lim_{m\to\infty}\overline{\mu}(L\cap A^m)=\delta_{\overline{\mu}}(L),\]
and 
\[\lim_{n\to\infty}v_{m,n}=\lim_{m\to\infty}\overline{\mu}(L\cap A^m)\]
uniformly in $m$. By Lemma \ref{lemmaDoubleSequences}, the sequence $v_{m,n}$ converges
to $\delta_{\overline{\mu}}(L)$.
\end{proofof}

Theorem \ref{theoremMain} is, however,
of limited application. 
In fact, a sequence of ergodic measures never
converges in the strong sense unless
it is eventually constant. 
Indeed, let $\mu_n$ be a sequence of ergodic measures
converging in the strong sense to $\overline{\mu}$. If $\mu_n\ne\overline{\mu}$
for all $n\ge 0$,
the measures $\mu_n$ and $\mu$ being mutually singular, the sequence $(\mu_n)$
cannot converge in the strong sense to $\overline{\mu}$.

%%%%%%%%%%%%%%%%%%%%%%%%%
\section{Densities in more general cases}\label{sectionWeak}

We give two examples of interest.
We start with an example for which
the sequence converges, but not to
a positive measure.

Consider the shift of finite type on $A=\{a,b,c\}$ with set of forbidden blocks
\[W_n=\{aA^{i}a: i \leq n \} \]
and call this shift $X_n$ for each $n\ge 2$.
Observe that $X_n$ is the elements of $A^{\Z}$ in which occurrences of $a$ are at least $n$ digits apart.
We note also that $X_n \supsetneq X_{n+1}$ for each $n$, and each is a shift of finite type.
Therefore, we may define $\mu_n$ as the measure of maximal entropy on the support of $X_n$, and it is clear that $\mu_n$ is a Markov measure.

The shift $X=\cap_{n\in \N} X_n$ is
constituted of the sequences
having at most one occurrence
of the letter $a$. 
We consider the rational language $L=\{b,c\}^*a\{b,c\}^*$.
The set  $L\cap A^n$ consists of the $n$ sequences of
length $n$ over $A$ having only one $a$. Therefore, $\mu_n(L\cap A^n)$
tends to $0$.

The sequence $(\mu_n)_{n \in \N}$ converges to the measure of maximal entropy on $\{b,c\}^{\Z} \bigcup \{b,c\}^{-\N}a\{b,c\}^{\N}$. Call this measure $\bar{\mu}$.
It is clear that $\bar{\mu}$ is
the Bernoulli measure given by $\bar{\mu}(a)=0$,
$\bar{\mu}(b)=
\bar{\mu}(c)=\frac{1}{2}$.
We see that $\delta_{\overline{\mu}}(L)=0$.
The sequential density of $L$
is therefore equal to its density with respect to $\bar{\mu}$.
On the other hand, if we consider the
rational language $L'=\cL(X)=L\cup\{b,c\}^*$,
then $0=\delta_\mu(L')\ne\delta_{\overline{\mu}}(L')=1$.\\

We give a second example showing that Lynch's theorem might fail to hold
for  probability measures that are not Markov  measures.

For $n\ge 2$, let $X_n$ be  the shift of finite type on $A=\{a,b,c\}$ with the following set of forbidden blocks:
\[W_n=\begin{cases}aA^{n-2}a&\mbox{if $\lfloor \log_2(n)\rfloor$ is even,}\\
aA^{n-2}b&\mbox{otherwise.}\end{cases}\]
The shift $X_n$ is represented for $n=2$ in Figure~\ref{figureShiftXn}.
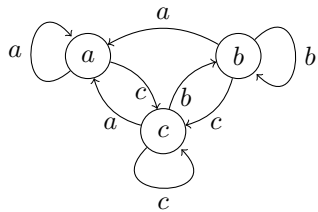
\begin{figure}[hbt]
\centering
\tikzset{node/.style={circle,draw,minimum size=.6cm,inner sep=0pt}}
    \tikzset{title/.style={minimum width=.4cm,minimum height=.4cm}}
    \tikzstyle{every loop}=[->,shorten >=1pt,looseness=6]
    \tikzstyle{loop left}=[in=130,out=220,loop]
    \tikzstyle{loop right}=[in=-45,out=45,loop]
    \tikzstyle{loop below}=[in=-45,out=-135,loop]
\begin{tikzpicture}
%n=2
\node[node](a)at(0,0){$a$};
\node[node](b)at(2,0){$b$};
\node[node](c)at(1,-1){$c$};

\draw[left,loop left,->](a)edge node{$a$}(a);
\draw[right,loop right,->](b)edge node{$b$}(b);
\draw[below,loop below,->](c)edge node{$c$}(c);
\draw[above,bend right,->](b)edge node{$a$}(a);
\draw[below,bend left,->](c)edge node{$a$}(a);
\draw[below,bend left,->](a)edge node{$c$}(c);
\draw[below,bend left,->](b)edge node{$c$}(c);
\draw[below,bend left,->](c)edge node{$b$}(b);

\end{tikzpicture}
\caption{The shift $X_2$.}\label{figureShiftXn}
\end{figure}

Each shift $X_n$ is irreducible. 
Indeed, $X_n$ is an edge shift on the graph $\Gamma_n$ with $A^{n-1}$ as the set of vertices, where the edges are the
words of length $n$ not in $W_n$. 
The vertex $a^{n-1}$ is connected to $c^{n-1}$, which is connected to all words of length $n-1$.

Let $\mu_n$ be the measure of maximal entropy on $X_n$. Then $(\mu_n)$ converges to the uniform Bernoulli
distribution on the full shift $\{a,b,c\}^\Z$.
The measure $\mu_2$ is the measure computed in Example~\ref{exampleMeasureMaxEntropy}.

We claim that the sequence $(\mu_n)$ is a sequence of ergodic measures converging to the
uniform Bernoulli measure and that the sequential density of the language $L=aA^*\cap A^*a$ does
not exist.

Indeed, each $\mu_n$ is an irreducible stationary Markov chain, which implies that it is ergodic.
It converges to the uniform Bernoulli measure because the sequence $(X_n)$ converges to
the full shift $A^\Z$. Next, we have
\[\frac{1}{n}\sum_{i=0}^{n-1}\mu_i(L\cap A^i)<\frac{1}{2}
\]
if $n=2^{2k+1}$ and 
\[\frac{1}{n}\sum_{i=0}^{n-1}\mu_i(L\cap A^i)>\frac{1}{2}+\frac{1}{8}=\frac{5}{8}\]
otherwise. Thus the limit does not exist.

\section{Combinatorial density}\label{sectionCombinatorial}
Let $X$ be a shift space on the alphabet $A$. The \emph{combinatorial density}
of a language $L\subset A^*$ is the limit
\[\delta_X(L)=\lim_{n\to\infty}\frac{1}{n}\sum_{i=0}^{n-1}\frac{\Card(L\cap\cL_i(X))}{\Card(\cL_i(X)}\]
provided the limit exists.

The combinatorial density is a particular case of sequential density. 
For each $n\ge 0$, let $\mu_n$ be
a probability measure on $A^\Z$ such that its $n^{th}$ marginal distribution for $w\in A^n$ is the following:
\[\mu_n(w)=\begin{cases}\frac{1}{\Card(\cL_n(X)}&\mbox{if $w\in\cL_n(X)$},\\0&\mbox{otherwise.}\end{cases}\]
Then, 
\[\mu_n(L)=\frac{\Card(L\cap\cL_n(X)}{\Card(\cL_n(X)}.\]
Therefore, the combinatorial density of $L$ is equal to its sequential density with respect to
the sequence $(\mu_n)$.

 The following shows that, when they exist, the combinatorial density
of  right ideals is its density with respect to some probability measure on $X$.
\begin{proposition}
Let $X$ be a shift space such that $\delta_X(uA^*)$ exists for each $u\in\cL(X)$. There is a probability measure $\mu$ on $X$ such that
for every $u\in\cL(X)$ one has
\[\delta_X(uA^*)=\mu(u).\]
\end{proposition}
\begin{proof}
One has for every $u\in A^*$
\begin{eqnarray*}
\sum_{a\in A}\delta_X(uaA^*)&=&\sum_{a\in A}\lim_{n\to\infty}\frac{1}{n}\sum_{i=0}^{n-1}\frac{\Card(uaA^*\cap \cL_i(X))}{\Card(\cL_i(X))}\\
&=&\lim_{n\to\infty}\frac{1}{n}\sum_{i=0}^{n-1}\frac{\sum_{a\in A}\Card(uaA^*\cap \cL_i(X))}{\Card(\cL_i(X))}\\
&=&\lim_{n\to\infty}\frac{1}{n}\sum_{i=0}^{n-1}\frac{\Card(uA^*\cap A^i)}{\Card(\cL_i(X))}\\
&=&\delta_X(uA^*).
\end{eqnarray*}
Since $\delta_X(A^*)=1$, this shows that there is a unique Borel probability measure on $X$
such that $\mu([u])=\delta_X(uA^*)$ for every $u\in\cL(X)$.
\end{proof}
\begin{proposition}\label{propositionDensitySturm}
Let $X$ be a Sturmian shift on $A=\{0,1\}$. For every $u\in\cL(X)$, one has
\[\delta_X(uA^*)=\mu(u),\]
where $\mu$ is  the (unique) invariant probability measure on $X$.
\end{proposition}
\begin{proof}
Let $x\in A^{-\N}$ be the unique right-special sequence in $X^-$. 
 The integer
$\Card(uA^*\cap \cL_n(X))-1$ is equal to the number of right-special words
of length at most $n$ beginning with $u$. Since all right-special words are suffixes of $x$,
it is also equal to the number $f_n(x,u)$
of occurrences of $u$ in $x_{(-n,0]}$. Thus
\[\delta_X(uA^*)=\lim_{n\to\infty}\frac{1+f_n(x,u)}{n+1}=\mu(u).\]
\end{proof}

\subsection{Acknowledgements} The first-named author was supported by Cofund Math In Greater Paris, Marie Sk\l odowska-Curie Actions (H2020-MSCA-COFUND-2020-GA101034255).

\bibliographystyle{plain}
\bibliography{probasSequential.bib}
\end{document}